\documentclass[11pt]{amsart}

\usepackage{amssymb,epsfig}
\usepackage{amsmath,verbatim,psfrag,color}
\usepackage[hmargin=1.5in, vmargin=1.5in]{geometry}

\newcommand{\nc}{\newcommand}
\nc{\dmo}{\DeclareMathOperator}
\nc{\nt}{\newtheorem}

\nt{thm}{Theorem}[section]

\nc{\SI}{\mathcal{SI}}
\nc{\SIBK}{\mathcal{SIBK}}
\dmo{\Mod}{Mod}
\dmo{\SMod}{SMod}

\nc{\Z}{\mathbb{Z}}

\nc{\p}[1]{\medskip\noindent {\bf #1.}}
\nc{\margin}[1]{\marginpar{\scriptsize #1}}

\title{Factoring in the hyperelliptic Torelli group}

\begin{document}
	
\author{Tara E. Brendle}

\author{Dan Margalit}

\address{Tara E. Brendle \\ School of Mathematics and Statistics \\15 University Gardens \\ University of Glasgow \\ G12 8QW \\ tara.brendle@glasgow.ac.uk}

\address{Dan Margalit \\ School of Mathematics\\ Georgia Institute of Technology \\ 686 Cherry St. \\ Atlanta, GA 30332 \\  margalit@math.gatech.edu}

\thanks{The first author is partially supported by the EPSRC.  The second author is supported by the Sloan Foundation and the National Science Foundation.}



\begin{abstract}
The hyperelliptic Torelli group is the subgroup of the mapping class group consisting of elements that act trivially on the homology of the surface and that also commute with some fixed hyperelliptic involution.  The authors and Putman proved that this group is generated by Dehn twists about separating curves fixed by the hyperelliptic involution.  In this paper, we introduce an algorithmic approach to factoring a wide class of elements of the hyperelliptic Torelli group into such Dehn twists, and apply our methods to several basic elements.
\end{abstract}

\maketitle


\section{Introduction}

Let $s : S_g^1 \to S_g^1$ be a hyperelliptic involution of a surface of genus $g$ with one boundary component; see Figure~\ref{figure:hi}.  The hyperelliptic Torelli group $\SI(S_g^1)$ is the group of homeomorphisms of $S_g^1$ that commute with $s$, restrict to the identity on $\partial S_g^1$, and act trivially on $H_1(S_g^1)$, modulo isotopy.  This group arises as the fundamental group of each component of the branch locus of the period mapping and also as the kernel of the Burau representation at $t=-1$; see \cite{hain}.

\definecolor{magenta4}{rgb}{.543,0,.543}

\begin{figure}[htb]
\psfrag{...}{\textcolor{magenta4}{$\dots$}}
\centerline{\includegraphics[scale=.65]{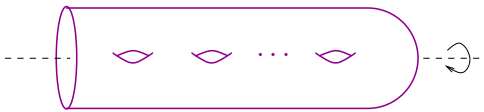}}
\caption{Rotation by $\pi$ about the indicated axis is a hyperelliptic involution.}
\label{figure:hi}
\end{figure}

The simplest nontrivial element of $\SI(S_g^1)$ is a Dehn twist about a symmetric separating curve, that is, a separating curve fixed by $s$.  Hain conjectured $\SI(S_g^1)$ is generated by such elements, and the authors recently proved this conjecture with Putman \cite{BMP}.    
There are two other basic elements of $\SI(S_g^1)$:
\begin{description}
 \item[Symmetric simply intersecting pair maps] If $x$ and $y$ are symmetric nonseparating curves with vanishing algebraic intersection $\hat\imath(x,y)$, then the commutator of their Dehn twists $[T_x,T_y]$ lies in $\SI(S_g^1)$; see Figure~\ref{figure:maps}.
\item[Symmetrized simply intersecting pair maps]  If $u_1$, $v_1$, $u_2$, and $v_2$ are nonseparating curves with $|u_1 \cap v_1|=2$, $\hat\imath(u_1,v_1)=0$, $s(u_1)=u_2$, $s(v_1)=v_2$, and $(u_1 \cup v_1) \cap (u_2 \cap v_2) = \emptyset$, then $[T_{u_1}T_{u_2},T_{v_1}T_{v_2}]$ lies in $\SI(S_g^1)$; see Figure~\ref{figure:maps}.
\end{description}
When the authors first learned of Hain's conjecture, it seemed intractable because we did not know how to factor these elements into Hain's proposed generators.  In this paper, we not only give relatively simple factorizations for both, but we also give an algorithm for factoring a much wider class of elements.  We expect that our relations will play a role for $\SI(S_g^1)$ analogous to the critical role that the classical lantern relation has played in our understanding of the full Torelli group.

\p{Higher genus twists} The genus of a separating curve in $S_g^1$ is the genus of the complementary component not containing $\partial S_g^1$.  In our earlier paper  \cite{sibk}, we showed that a Dehn twist about a symmetric separating curve of arbitrary genus is equal to a product of Dehn twists about symmetric separating curves of genus 1 and 2.  In particular, by our theorem with Putman, $\SI(S_g^1)$ is generated by Dehn twists about such curves.  As an application of the methods of this paper, we give an explicit factorization of the Dehn twist about any genus $k \geq 3$ symmetric separating curve into Dehn twists about symmetric separating curves of smaller genus.

\begin{figure}[htb]
\psfrag{a}{$y$}
\psfrag{b}{$x$}
\psfrag{u}{$u_1$}
\psfrag{v}{$v_1$}
\psfrag{u'}{$u_2$}
\psfrag{v'}{$v_2$}
\centerline{\includegraphics[scale=.9]{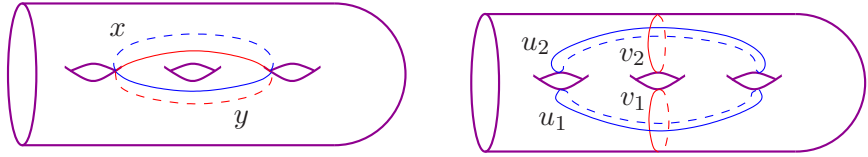}}
\caption{Left: The curves $x$ and $y$ form a symmetric simply intersecting pair. Right: the curves $u_1$, $v_1$, $u_2$, and $v_2$ form a symmetrized simply intersecting pair.}
\label{figure:maps}
\end{figure}

\p{Algorithmic factorizations} Let $a$ be a symmetric nonseparating curve in $S_g^1$, and denote by $\SI(S_g^1,a)$ the stabilizer of the isotopy class of $a$ in $\SI(S_g^1)$.  There is an $s$-equivariant inclusion $S_g^1-a \to S_{g-1}^1$ and this induces a surjective homomorphism $\SI(S_g^1,a) \to \SI(S_{g-1}^1)$ \cite[Proposition 6.6]{sibk}.  We denote the kernel by $\SIBK(S_g^1,a)$:
\[ 1 \to \SIBK(S_g^1,a) \to \SI(S_g^1,a) \to \SI(S_{g-1}^1) \to 1. \]

\begin{thm}
\label{thm:alg}
There is an explicit algorithm for factoring arbitrary elements of $\SI(S_g^1,a)$ into Dehn twists about symmetric separating curves of genus 1 and 2.
\end{thm}

The idea is to identify $\SIBK(S_g^1,a)$ with a subgroup of the fundamental group of a disk with $2g-1$ points removed.  Then the problem of factoring elements of $\SIBK(S_g^1,a)$ into Dehn twists about symmetric separating curves is translated into a problem of finding special factorizations of certain elements of this free group.

Given any symmetric simply intersecting pair map or symmetrized simply intersecting pair map, we can find a curve $a$ so that the given map lies in $\SIBK(S_g^1,a)$; choose $a$ to be the core of any annular region in the complement of the union of the defining curves of the map.  Therefore, we can understand both types of maps in the context of Theorem~\ref{thm:alg}.

If $c$ is a genus $k$ symmetric separating curve in $S_g^1$, then we can choose a genus $k-1$ symmetric separating curve $d$ and a symmetric nonseparating curve $a$ so that $T_cT_d^{-1}$ lies in the corresponding $\SIBK(S_g^1,a)$; we take $a$ to lie in the genus 1 subsurface between $c$ and $d$.  Therefore, by Theorem~\ref{thm:alg}, we can factor $T_cT_d^{-1}$ into a product of Dehn twists about symmetric separating curves of genus 1 and 2.

We emphasize that the existence of this algorithm does not guarantee that one can find a simple factorization for a given element of $\SIBK(S_g^1,a)$.  The factorizations we give in this paper were only found after much trial and error (cf. \cite{v1}).  Their relative tameness suggests that $\SI(S_g^1)$ is more tractable than originally believed.

Finally, we can obtain factorizations in the hyperelliptic Torelli group of a closed surface $S_h$ by including $S_g^1$ into $S_h$ where $h \geq g$; the induced map $\SI(S_g^1) \to \SI(S_h)$ is injective if $h > g$ and has cyclic kernel $\langle T_{\partial S_g^1} \rangle$ otherwise; see  \cite[Theorem 4.2]{sibk}.

\p{Acknowledgments} First, we would like to thank Andrew Putman for suggesting the problem of factoring the symmetrized simply intersecting pair maps.  We would also like to thank Mladen Bestvina, Joan Birman, Leah Childers, and an anonymous referee for helpful comments and conversations.  Our work was greatly aided by the Maple program $\beta$-Twister, written by Marta Aguilera and Juan Gonz\'alez-Meneses.

\section{The factoring algorithm}
\label{sec:alg}

In this section we explain how to algorithmically factor an arbitrary element of $\SIBK(S_g^1,a)$ as per Theorem~\ref{thm:alg}.  

\p{The setup}\label{sec:bh}  We would like to rephrase our problem about factoring elements of $\SI(S_g^1,a)$ into a problem about certain factorizations in a free group. We start by giving some definitions and then outlining the idea.

Let $D_{2g-1}$ denote a disk with $2g-1$ marked points and $D_{2g-1}^\circ$ the punctured disk obtained by removing the marked points.  The fundamental group of $D_{2g-1}^\circ$ is a free group $F_{2g-1}$; we take the generators $x_1,\dots,x_{2g-1}$ for $F_{2g-1}$ to be simple loops in $D_{2g-1}$ each surrounding one marked point; see Figure~\ref{figure:zetas}.

\begin{figure}[htb]
\psfrag{z1}{$x_{2g-1}$}
\psfrag{z2}{$x_{2g-2}$}
\psfrag{z2g-1}{$x_1$}
\psfrag{abar}{$p$}
\psfrag{...}{$\ \, \cdots$}
\centerline{\includegraphics[scale=1.1]{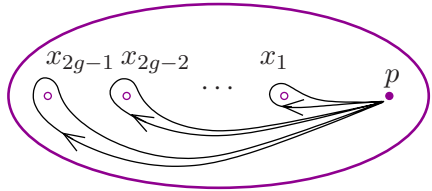}}
\caption{The generators $x_i$ for $\pi_1(D_{2g-1}^\circ)$.}
\label{figure:zetas}
\end{figure}

Let $F_{2g-1}^{even}$ denote the kernel of the map $F_{2g-1} \to \Z/2\Z$ given by $x_i \mapsto 1$ for all $i$; this group is generated by the $x_i^{\delta_i}x_j^{\delta_j}$ with $i \leq j$ and $\delta_i \in \{-1,1\}$.  Denote the generators for $\Z^{2g-1}$ by $e_1,\dots,e_{2g-1}$.  Let $\epsilon : F_{2g-1}^{even} \to \Z^{2g-1}$ be the homomorphism given by $x_i^{\delta_i} x_j^{\delta_j} \mapsto e_i - e_j$.
We will require the following two facts, explained below.
\begin{enumerate}
\item  There is an isomorphism $\Psi : \SIBK(S_g^1,a) \to \ker \epsilon$.
\item  $\ker \epsilon$ is generated by squares of simple loops in $D_{2g-1}^\circ$ about 1 or 3 punctures.
\end{enumerate}
Once we define $\Psi$, it will be easy to see that squares of simple loops in $D_{2g-1}^\circ$ surrounding 1 or 3 punctures correspond to products of Dehn twists about symmetric separating curves in $S_g^1$ of genus 1 and 2.
After discussing the above two facts, we proceed to explain the factorization algorithm of Theorem~\ref{thm:alg}.

\p{The isomorphism $\Psi$} The isomorphism $\Psi$ was given in our earlier paper \cite[Theorem 1.2]{sibk}; we recall the construction.  In what follows, the mapping class group of a surface $S$ is the group $\Mod(S)$ of isotopy classes of homeomorphisms of $S$ that restrict to the identity on $\partial S$ and preserve the set of marked points.  

The quotient $S_g^1/\langle s \rangle$ is a disk $D_{2g+1}$ with $2g+1$ marked points, and $\Mod(D_{2g+1})$ is isomorphic to the braid group $B_{2g+1}$.  Let $\SMod(S_g^1)$ be the subgroup of $\Mod(S_g^1)$ with elements represented by $s$-equivariant homeomorphisms.  Birman--Hilden proved the natural map $\Theta : \SMod(S_g^1) \to B_{2g+1}$ is an isomorphism \cite[Theorem 9.1]{primer}.

The group $\SI(S_g^1,a)$ maps to $\Mod(D_{2g+1},\overline a)$, the stabilizer of the isotopy class of the arc $\overline a$, the image of $a$ in $D_{2g+1}$.   By collapsing $\overline a$ to a marked point $p$ and removing the other $2g-1$ marked points to obtain $2g-1$ punctures, we obtain a homomorphism $\Xi : \Mod(D_{2g+1},\overline a) \to \Mod(D_{2g-1}^\circ,p)$.  Since the kernel of $\Xi \circ \Theta$ is generated by $T_a$, the restriction $\Psi : \SI(S_g^1,a) \to \Mod(D_{2g-1}^\circ,p)$ is injective.

We then arrive at the following special case of the Birman exact sequence:
\[
 1 \to \pi_1(D_{2g-1}^\circ,p) \to \Mod(D_{2g-1}^\circ,p) \to \Mod(D_{2g-1}^\circ) \to 1.
 \]
 The first nontrivial map here is actually an anti-homomorphism, as the usual orders of operation in the two groups do not agree.  Therefore, relations in $\pi_1(D_{2g-1}^\circ,p)$ will translate to the reverse relations in $\Mod(D_{2g-1}^\circ,p)$.
 
 The image of $\SIBK(S_g^1,a)$ under $\Psi$ lies in the kernel $\pi_1(D_{2g-1}^\circ,p)$ of the Birman exact sequence.  In our earlier paper \cite[Lemma 4.5]{sibk} we showed that for $\alpha \in \pi_1(D_{2g-1}^\circ,p)$, the action of the lift $(\Xi \circ \Theta)^{-1}(\alpha)$ on $H_1(S_g^1;\Z)$ is exactly given by $\epsilon(\alpha)$, and so the image of $\Psi$ is precisely $\ker \epsilon$.

\p{Squares of simple loops} If $\alpha \in \pi_1(D_{2g-1}^\circ,p)$ is a simple loop surrounding $k$ punctures, where $k$ is odd, then $\alpha^2$ lies in $\ker \epsilon = \textrm{Im} \Psi$ and  $\Psi^{-1}(\alpha^2)$ is equal to $T_cT_d^{-1}$, where $c$ and $d$ are the preimages in $S_g^1$ of the curves obtained by pushing $\alpha$ off of $p$ to the left and right, respectively, and positive Dehn twists are to the left.  The curves $c$ and $d$ are separating curves of genus $(k+1)/2$ and $(k-1)/2$ (not necessarily in that order) and $a$ lies in the genus 1 subsurface between them.  When $k=1$, note that one of the two separating curves is inessential.

We will now show that these $\alpha^2$ generate $\ker \epsilon$.  To begin, the image of $\epsilon$ is $\Z_{bal}^{2g-1}$, the kernel of the map $\Z^{2g-1} \to \Z$ recording  the coordinate sum \cite[Lemma 5.1]{sibk}.  Also, the group $F_{2g-1}^{even}$ is generated by elements of the form $x_i^2$ and the $x_jx_1$, since
\[  x_i x_j = (x_i x_1) (x_jx_1)^{-1}(x_j^2) \quad \text{and} \quad x_i x_j^{-1} = (x_i x_1) (x_jx_1)^{-1}, \]
and $\Z_{bal}^{2g-1}$ has a presentation whose generators are the images of these generators:
\[ \langle \epsilon(x_1^2), \dots, \epsilon(x_{2g-1}^2), \epsilon(x_2x_1) , \dots, \epsilon(x_{2g-1}x_1) \mid \epsilon(x_i^2), [\epsilon(x_ix_1),\epsilon(x_jx_1)] \rangle. \]
It follows that $\ker \epsilon$ is normally generated by the set
\[ \{ x_i^2 \mid 1 \leq i \leq 2g-1 \} \cup \{ [x_ix_1,x_jx_1] \mid 1 \leq i < j \leq 2g-1 \}. \]
We notice the following relation in $\ker \epsilon$:
\[ [x_ix_1,x_jx_1] = [x_j^{-2}(x_jx_ix_1)^2(x_i^{-2})^{x_1^{-1}}x_1^{-2}]^{x_j}, \]
where $x^y$ denotes $yxy^{-1}$.  It now follows that $\ker \epsilon$ is normally generated by
\[ \{x_i^2 \mid 1 \leq i \leq 2g-1 \} \cup \{ (x_jx_ix_1)^2 \mid 1 \leq i < j \leq 2g-1 \}. \]
Referring to Figure~\ref{figure:zetas}, we see that each $x_jx_ix_1$ is a simple closed curve in $D_{2g-1}^\circ$ when $j > i$.  In particular, $\ker \epsilon$ is generated by
\[ \{ \alpha^2 \mid \alpha \text{ is a simple loop surrounding 1 or 3 punctures} \}. \]
It follows that $\SIBK(S_g^1,a)$ is generated by maps of the form $T_c$ where $c$ is a symmetric separating curve of genus 1 with $a$ lying on the genus 1 side of $c$ and of the form $T_cT_d^{-1}$ where $c$ and $d$ are symmetric separating curves of genus 1 and 2, respectively, with $a$ lying in the genus 1 subsurface between.

\p{The algorithm}  We now give an algorithm for factoring arbitrary elements of $\ker \epsilon$ in terms of squares of simple loops in $D_{2g-1}^\circ$, each surrounding 1 or 3 punctures.

Suppose we are given some $f \in \SIBK(S_g^1,a)$ as a product of Dehn twists about symmetric curves in $S_g^1$.  We can realize $f$ as an element $\overline f$ of the group $\Mod(D_{2g-1}^\circ,p)$ using the following dictionary: a Dehn twist about a symmetric nonseparating curve $c$ corresponds to a half-twist about the image arc $\overline c$ in $D_{2g-1}^\circ$ and a Dehn twist about a symmetric separating curve corresponds to the square of the Dehn twist about the image curve in $D_{2g-1}^\circ$.  As  above, since $f \in \SIBK(S_g^1, a)$, we know that $\overline f$ lies in the kernel of the above Birman exact sequence.   We can then use Artin's combing algorithm for pure braids \cite{artin2} to write $\overline f$ as a word $w_0$ in the $x_i^{\pm 1}$.

As above, the word $w_0$ lies in $F_{2g-1}^{even}$, and so it equals some word $w$ in the $x_i^2$ and the $x_i x_1$. 
Since $\Psi(f) \in \ker \epsilon$, the word $w$ maps to a relator in $\Z^{2g-1}_{bal}$ with respect to the presentation given above, that is, $w$ maps to a word in the generators $\epsilon(x_i^2)$ and $\epsilon(x_ix_1)$ that equals the identity.  Therefore, there is a sequence of commutations, free cancellations, and cancellations of $\epsilon(x_i^2)$-terms transforming $\epsilon(w)$ into the empty word (this is the obvious solution to the word problem for a free abelian group).  By the correspondence between relators for $\Z^{2g-1}_{bal}$ and normal generators for $\ker \epsilon$, we obtain a factorization of $w$ into a product of conjugates of the $x_i^2$, the $[x_ix_1,x_jx_1]$, and their inverses.  We already explained above how to factor $[x_ix_1,x_jx_1]$ into a product of squares of simple loops surrounding 1 or 3 punctures, so we are done.

\p{Relations in genus two} Above, we explained how the commutator $[x_ix_1,x_jx_1]$ corresponds to an element of $\SI(S_2^1)$ and we factored this into a product of five Dehn twists about symmetric separating curves in $S_2^1$ (see \cite[Theorem 3.1]{v1} for a picture of the curves).  If we cap the boundary of $S_2^1$ with a disk, $[x_ix_1,x_jx_1]$ maps to $[T_c^2T_e^{-2},T_b^2T_d^{-2}]T_a^8$ in $\SI(S_2)$; see Figure~\ref{figure:crazy1}.  The related (but simpler) element $[T_cT_e^{-1},T_bT_d^{-1}]T_a^2$ also lies in $\SI(S_2)$ and so is a product of Dehn twists about symmetric separating curves.  Surprisingly, it equals a single (left) Dehn twist.

\begin{figure}[htb]
\psfrag{a}{$d$}
\psfrag{b}{$c$}
\psfrag{c}{$b$}
\psfrag{d}{$e$}
\psfrag{e}{$a$}
\psfrag{f}{$f$}
\centerline{\includegraphics[scale=.8]{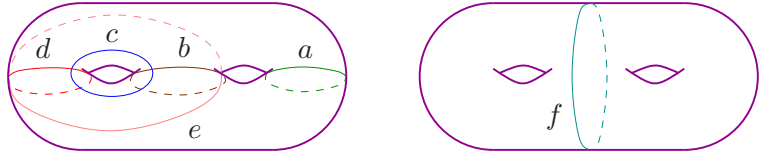}}
\caption{The curves $a$, $b$, $c$, $d$, $e$, and $f$ from Theorem~\ref{thm:aux1}.}
\label{figure:crazy1}
\end{figure}

\begin{thm}
\label{thm:aux1}
Let $a$, $b$, $c$, $d$, $e$, and $f$ be as in Figure~\ref{figure:crazy1}.  We have:
\[[T_cT_e^{-1},T_bT_d^{-1}]T_a^2 = T_f. \]
\end{thm}

One can check the relation in Theorem~\ref{thm:aux1} using the Alexander Method \cite[Section 2.3]{primer}; see \cite[Section 4.1]{v1} for a conceptual proof.


\section{Applications}
\label{sec:app}

In this section, we give explicit factorizations of symmetric simply intersecting pair maps and symmetrized simply intersecting pair maps into Dehn twists about symmetric separating curves.  We also give an explicit factorization of the Dehn twist about any genus $k \geq 3$ symmetric separating curve into Dehn twists about symmetric separating curves of smaller genus.

\subsection{Factoring symmetric simply intersecting pair maps}
\label{sec:ssip}

We start by writing the symmetric simply intersecting pair map from Figure~\ref{figure:ssip} as a product of Dehn twists about symmetric separating curves.  

\begin{figure}[htb]
\psfrag{x}{$a$}
\psfrag{a}{$y$}
\psfrag{b}{$x$}
\psfrag{c}{$v$}
\psfrag{c'}{$w$}
\centerline{\includegraphics[scale=.9]{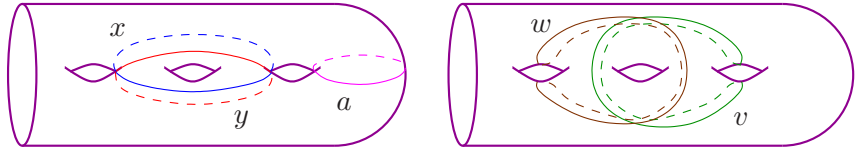}}
\caption{The curves used in Theorem~\ref{thm:ssip}.}
\label{figure:ssip}
\end{figure}

\begin{thm}
\label{thm:ssip}
Every symmetric simply intersecting pair map is the product of two Dehn twists about symmetric simple closed curves.  In particular, 
if $x$, $y$, $v$, and $w$ are the simple closed curves shown in Figure~\ref{figure:ssip}, we have:
\[ [T_x,T_y] = T_v^{-1}T_w. \]
\end{thm}

Note that the first statement follows immediately from the second statement and the change of coordinates principle \cite[Section 1.3]{primer}.

We give two proofs of Theorem~\ref{thm:ssip}.  The first is an easy application of the lantern relation, a relation between (left) Dehn twists about 7 curves lying in a subsurface homeomorphic to a sphere with four boundary components;  see \cite[Section 5.1.1]{primer}.

\begin{proof}[First proof of Theorem~\ref{thm:ssip}]

By the lantern relation, we have $T_vT_xT_y=M$ and $T_{w}T_yT_x=M'$, where $M$ and $M'$ are the products of twists about the boundary curves of the corresponding four-holed spheres.  Since $x$ and $y$ appear in both lantern relations and since a regular neighborhood of $x \cup y$ is a sphere with four holes, the four-holed spheres in the two lantern relations are equal, and so $M=M'$.  Thus,
\[
[T_x,T_y] = (T_xT_y)(T_x^{-1}T_y^{-1}) = (T_{v}^{-1}M)(M^{-1}T_{w}) = T_{v}^{-1}T_{w},\]
as desired.
\end{proof}

We now give a proof of Theorem~\ref{thm:ssip} that is intrinsic to the braid group.

\begin{figure}[htb]
\psfrag{a}{$\beta$}
\psfrag{b}{$\alpha$}
\psfrag{x}{$\overline x$}
\psfrag{y}{$\overline y$}
\psfrag{w}{$\overline v$}
\psfrag{w'}{$\overline w$}
\psfrag{g}{$\gamma$}
\psfrag{de}{$\delta$}
\psfrag{d}{$d$}
\psfrag{e}{$e$}
\psfrag{p}{$q_1$}
\psfrag{q}{$q_2$}
\centerline{\includegraphics[scale=.9]{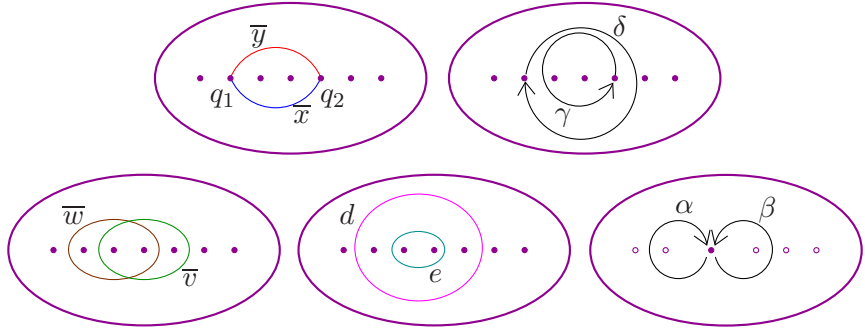}}
\caption{Curves, loops, and arcs used in the second proof of Theorem~\ref{thm:ssip}}
\label{figure:wbar}
\end{figure}

\begin{proof}[Second proof of Theorem~\ref{thm:ssip}]

The images of $x$ and $y$ in $D_{2g+1}$ are arcs $\overline x$ and $\overline y$; denote their endpoints by $q_1$ and $q_2$ (throughout, refer to Figure~\ref{figure:wbar}).  As above, $T_x$ and $T_y$ correspond to the half-twists $H_{\overline x}$ and $H_{\overline y}$ in $\Mod(D_{2g+1})$.  

As a loop in the space of configurations of $2g+1$ points in the disk (see \cite[Theorem 9.1]{primer}), the product $H_{\overline x}H_{\overline y}$ is given by the motion of points where $q_1$ and $q_2$ move around $\delta$ and $\gamma$, respectively (we multiply half-twists right to left).  These motions correspond to the mapping classes $T_{\overline v}^{-1}T_{d}$ and $T_{\overline v}^{-1}T_e$, respectively.  Similarly, $H_{\overline x}^{-1}H_{\overline y}^{-1}$ corresponds to $(T_{\overline w}T_d^{-1})(T_{\overline w}   T_e^{-1})$.  Since $T_d$ and $T_e$ commute with all the other twists, the original commutator $[T_x,T_y]$ in $\SI(S_g^1)$ corresponds to $T_{\overline v}^{-2}T_{\overline w}^{2}$ in $\Mod(D_{2g+1})$.  The preimage under $\Psi$ is $T_{v}^{-1}T_{w}$ in $\SMod(S_g^1)$, as desired.
\end{proof}

The second proof of Theorem~\ref{thm:ssip} has a connection with the algorithm from Section~\ref{sec:alg}.  Let $a$ be the symmetric simple closed curve in $S_g^1$ shown in Figure~\ref{figure:ssip}.  Assuming Theorem~\ref{thm:ssip}, and referring to Figure~\ref{figure:wbar}, we can see that $T_{\overline v}^{-2}T_{\overline w}^{2}$, the image of $T_v^{-1}T_{w}$ under $\Psi$, is $\alpha^2\beta^2$, where $\alpha$ and $\beta$ are as shown in Figure~\ref{figure:wbar}.  This is a product of squares of simple loops, each surrounding one puncture, as per Section~\ref{sec:alg}.


\subsection{Factoring symmetrized simply intersecting pair maps}
\label{sec:szsip}

We now address the symmetrized simply intersecting pair maps.

\begin{thm}
\label{thm:szsip}
Every symmetrized simply intersecting pair map is equal to a product of six Dehn twists about symmetric separating curves.
\end{thm}

\begin{figure}[htb]
\psfrag{ubar}{$\overline u$}
\psfrag{vbar}{$\overline v$}
\psfrag{u}{$u_1$}
\psfrag{v}{$v_1$}
\psfrag{u'}{$u_2$}
\psfrag{v'}{$v_2$}
\psfrag{x}{$a$}
\psfrag{a}{$\alpha$}
\psfrag{b}{$\beta$}
\psfrag{c}{$\gamma$}
\psfrag{de}{$\delta$}
\psfrag{d}{$d$}
\psfrag{e}{$e$}
\centerline{\includegraphics[scale=.9]{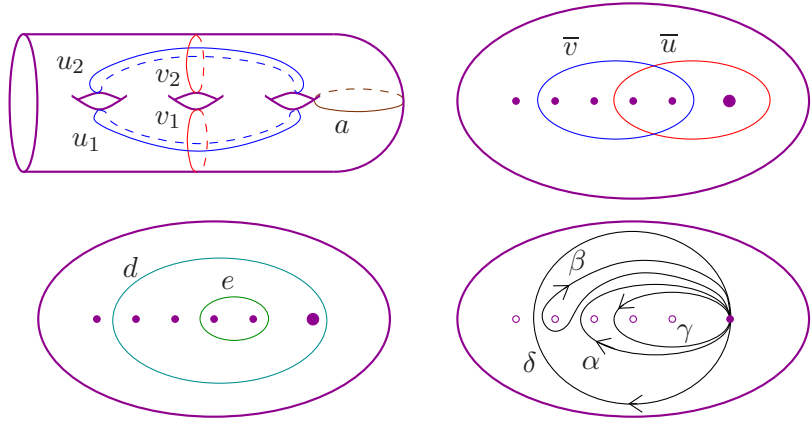}}
\caption{The curves and loops used in the proof of Theorem~\ref{thm:szsip}.}
\label{figure:cibar}
\end{figure}

\begin{proof}

Consider the symmetrized simply intersecting pair map shown in Figure~\ref{figure:cibar} (throughout we refer to this figure).  First we notice that $[T_{u_1}T_{u_2},T_{v_1}T_{v_2}]$ lies in $\SIBK(S_g^1,a)$.  We claim that the image of this commutator under the map $\Psi$ from Section~\ref{sec:alg} is $[\delta,\gamma]$.  Indeed, we have $\Psi(T_{u_1}T_{u_2})=T_{\overline u}$ and $\Psi(T_{v_1}T_{v_2})=T_{\overline v}$, and so the claim follows from the fact that the images of $\delta$ and $\gamma$ in $\Mod(D_{2g-1}^\circ,p)$ are $T_{\overline v}^{-1}T_{\overline v'}$ and $T_{\overline u}^{-1}T_{\overline u'}$ and the fact that $T_{\overline u'}$ and $T_{\overline v'}$ commute with all other twists in the commutator (remember that the order of multiplication gets reversed!).

Now that we have written $\Psi([T_{u_1}T_{u_2},T_{v_1}T_{v_2}])$ as an element of the free group $\pi_1(D_{2g-1}^\circ,p)$, we observe the following factorization in this free group:
\[ [\delta,\gamma] = [\beta^2 (\beta^{-1}\gamma)^2 (\alpha\gamma)^{-2} \alpha^2]^\alpha \]

As in Section~\ref{sec:alg}, this is a product of squares of simple loops in $\pi_1(D_{2g-1}^\circ,p)$ surronding 1 or 3 punctures, and hence the preimage under $\Psi$ is a product of Dehn twists about symmetric separating curves of genus 1 and 2 in $S_g^1$.  The first and fourth loops each correspond to a single Dehn twist, while the second and third loops each correspond to a pair of Dehn twists.
\end{proof}

From the proof of Theorem~\ref{thm:szsip}, it is straightforward, though not necessarily enlightening, to draw the six symmetric separating simple closed curves whose Dehn twists factorize the symmetrized simply intersecting pair map shown in Figure~\ref{figure:maps}.   For an explicit picture, see the first version of this paper \cite{v1}. 

In terms of the $x_i$ from Figure~\ref{figure:zetas}, we can also write the $\Psi$-image of a symmetrized simply intersecting pair map as $[x_4x_3,x_2x_1]$.  In the free group, this factors as:
\[ [x_4x_3,x_2x_1] = [(x_3x_2x_1)^2 (x_3^{-2})^{(x_2x_1)^{-1}} (x_4x_2x_1)^{-2}(x_4^2)]^{x_4}. \]
Again, the right-hand side in this equality is a product of squares of simple loops and so we obtain an alternate factorization.

The relation given in Theorem~\ref{thm:szsip} involves 14 Dehn twists, twice the number of Dehn twists involved in the lantern relation.  However, there is no way to rearrange our relation into a product of two lantern relations.


\subsection{Factoring higher genus twists}

Finally, we obtain the factorization of a Dehn twist about an arbitrary symmetric separating curve into a product of Dehn twists about symmetric separating curves, each having genus 1 or 2, by applying the following theorem inductively. 

\begin{thm}
\label{thm:genus k}
Let $d$ denote the boundary of $S_g^1$, and let $c$ denote a symmetric separating curve of genus $g-1$.  The product $T_dT_c^{-1}$ is equal to a product of 10 Dehn twists about symmetric separating curves in $S_g^1$, each of genus at most $g-1$.
\end{thm}

\begin{proof}

Let $a$ be a symmetric nonseparating curve in $S_g^1$ lying between $c$ and $d$.  The image of $d$ in $D_{2g+1}$ is the boundary of the disk, and if we choose the identification of $S_g^1/\langle s \rangle$ with $D_{2g+1}$ appropriately, the image of $c$ in $D_{2g+1}$ is a round circle surrounding the $2g-1$ leftmost marked points and the image of $a$ is a straight arc connecting the other two marked points.  

Let $y_5$ denote $x_{2g-1}x_{2g} \cdots x_5$.  It follows from the previous paragraph that the image of $T_dT_c^{-1}$ under $\Psi$ is equal to the image of $(y_5x_4x_3x_2x_1)^2$ under the point pushing map $\pi_1(D_{2g-1}^\circ,p) \to \Mod(D_{2g-1},p)$.

As in Section~\ref{sec:alg}, we factor $(y_5x_4x_3x_2x_1)^2$ into a product of simple loops each surrounding an odd number of punctures:
\begin{align*}
 [  (x_2x_1y_5)^2(x_2^{-2})^{y_5^{-1}x_1^{-1}}(x_3x_1y_5)^{-2}x_3^2]^{y_5x_4x_3}  
[ (x_1y_5x_4)^2  (x_4^{-2})(x_4x_3x_2)^2 ]^{x_1^{-1}}.
\end{align*}
This is a product of 11 Dehn twists about symmetric separating curves, three of genus $g-1$, three of genus $g-2$, four of genus 1, and one of genus 2. 
\end{proof}


\bibliographystyle{plain}
\bibliography{factor}

\def\cprime{$'$} \def\cprime{$'$}
\begin{thebibliography}{1}

\bibitem{artin2}
Emil Artin.
\newblock Theory of braids.
\newblock {\em Ann. of Math. (2)}, 48:101--126.

\bibitem{v1}
Tara Brendle and Dan Margalit.
\newblock {Factoring in the hyperelliptic Torelli group}.
\newblock {arXiv:1202.2365v1}.

\bibitem{sibk}
Tara Brendle and Dan Margalit.
\newblock {Point pushing, homology, and the hyperelliptic involution}.
\newblock {to appear in Michigan Mathematical Journal}.

\bibitem{BMP}
Tara Brendle, Dan Margalit, and Andrew Putman.
\newblock {Generators for the hyperelliptic Torelli group and the kernel of the
  Burau representation at $t=-1$}.
\newblock {Preprint}.

\bibitem{primer}
Benson Farb and Dan Margalit.
\newblock {\em {A primer on mapping class groups}}.
\newblock {Princeton Univ. Press}, {2011}.

\bibitem{hain}
Richard Hain.
\newblock Finiteness and {T}orelli spaces.
\newblock In {\em Problems on mapping class groups and related topics},
  volume~74 of {\em Proc. Sympos. Pure Math.}, pages 57--70. Amer. Math. Soc.,
  Providence, RI, 2006.

\end{thebibliography}

\end{document}